\documentclass[11pt,reqno]{amsart}
\calclayout
\usepackage{amsmath}
\usepackage{amssymb,amsthm,xypic,stmaryrd,graphicx,mathtools}
\usepackage[all]{xy}
\usepackage{bbm, dsfont}
\usepackage{boxedminipage,xcolor}
\usepackage[shortlabels]{enumitem}
\usepackage{thmtools}
\usepackage{thm-restate}
\usepackage{tikz-cd}
\usepackage{empheq}
\usepackage[colorlinks=true,linkcolor=blue,citecolor=blue]{hyperref}
\usetikzlibrary{arrows.meta}
\tikzcdset{arrow style=tikz}
\numberwithin{equation}{section}
\DeclareMathOperator{\id}{id}

\DeclareMathOperator{\Spin}{Spin}

\newcommand{\beq}[1]{\begin{equation} \label{#1}}
\newcommand{\eeq}{\end{equation}}

\newcommand{\bea}{\begin{eqnarray}}
\newcommand{\eea}{\end{eqnarray}} 

\newcommand{\Mod}[1]{\ (\mathrm{mod}\ #1)}

\begin{document}

\theoremstyle{plain}
\newtheorem{theorem}{Theorem}[section]
\newtheorem{thm}{Theorem}[section]
\newtheorem{lemma}[theorem]{Lemma}
\newtheorem{proposition}[theorem]{Proposition}
\newtheorem{prop}[theorem]{Proposition}
\newtheorem{corollary}[theorem]{Corollary}
\newtheorem{conjecture}[theorem]{Conjecture}
\newtheorem{question}[theorem]{Question}

\theoremstyle{definition}
\newtheorem{convention}[theorem]{Convention}
\newtheorem{definition}[theorem]{Definition}
\newtheorem{defn}[theorem]{Definition}
\newtheorem{example}[theorem]{Example}
\newtheorem{remark}[theorem]{Remark}
\newtheorem*{remark*}{Remark}
\newtheorem*{overview*}{Overview}
\newtheorem*{results*}{Results}
\newtheorem{rem}[theorem]{Remark}

\newcommand{\C}{\mathbb{C}}
\newcommand{\R}{\mathbb{R}}
\newcommand{\Z}{\mathbb{Z}}
\newcommand{\N}{\mathbb{N}}
\newcommand{\Q}{\mathbb{Q}}

\newcommand{\Supp}{{\rm Supp}}
\newcommand{\tn}{\textnormal}
\newcommand{\field}[1]{\mathbb{#1}}
\newcommand{\bZ}{\field{Z}}
\newcommand{\bR}{\field{R}}
\newcommand{\bC}{\field{C}}
\newcommand{\bN}{\field{N}}
\newcommand{\bT}{\field{T}}
\newcommand{\cB}{{\mathcal{B} }}
\newcommand{\cK}{{\mathcal{K} }}
\newcommand{\cF}{{\mathcal{F} }}
\newcommand{\cO}{{\mathcal{O} }}
\newcommand{\cE}{\mathcal{E}}
\newcommand{\cS}{\mathcal{S}}
\newcommand{\cN}{\mathcal{N}}
\newcommand{\calL}{\mathcal{L}}

\newcommand{\KK}{K \! K}

\newcommand{\norm}[1]{\| #1\|}

\newcommand{\Spinc}{\Spin^c}

\newcommand{\HH}{{\mathcal{H} }}
\newcommand{\Hpi}{\HH_{\pi}}

\newcommand{\DNR}{D_{N \times \R}}


\def\kt{\mathfrak{t}}
\def\kk{\mathfrak{k}}
\def\kp{\mathfrak{p}}
\def\kg{\mathfrak{g}}
\def\kh{\mathfrak{h}}
\def\so{\mathfrak{so}}
\def\cut{c}

\newcommand{\ddt}{\left. \frac{d}{dt}\right|_{t=0}}

\newcommand{\todo}{\textbf{TO DO}}

\title[Poincar\'{e} duality for actions by matrix groups]{An equivariant Poincar\'{e} duality for proper cocompact actions by matrix groups}

\author{Hao Guo}
\address[Hao Guo]{ Department of Mathematics, Texas A\&M University }
\email{haoguo@math.tamu.edu}

\author{Varghese Mathai}
\address[Varghese Mathai]{ School of Mathematical Sciences, University of Adelaide }
\email{mathai.varghese@adelaide.edu.au}


\maketitle

\begin{abstract}
Let $G$ be a linear Lie group acting properly on a $G$-spin$^c$ manifold $M$ with compact quotient. We give a short proof that Poincar\'{e} duality holds between $G$-equivariant $K$-theory of $M$, defined using finite-dimensional $G$-vector bundles, and $G$-equivariant $K$-homology of $M$, defined through the geometric model of Baum and Douglas.
\end{abstract}
\vspace{1cm}
\section{Introduction}
\label{sec intro}
\noindent Poincar\'{e} duality in $K$-theory asserts that the $K$-theory group of a closed spin$^c$ manifold is naturally isomorphic to its $K$-homology group via cap product with the fundamental class in $K$-homology. This class can be represented geometrically by the spin$^c$-Dirac operator. More generally, if a compact Lie group acts on the manifold preserving the spin$^c$ structure, the analogous map implements Poincar\'{e} duality between the equivariant versions of $K$-theory and $K$-homology. In the case when the Lie group is non-compact but has finite component group, induction on $K$-theory and $K$-homology allow one to establish the analogous result \cite{GMW} for proper actions. The  observation underlying Poincar\'{e} duality in all of these cases is that there exist enough equivariant vector bundles with which to pair the fundamental class.

In contrast, Phillips \cite{Phillips} showed through a counter-example with a non-linear group that, for proper actions by a general Lie group $G$ on a manifold $X$ with compact quotient space, finite-dimensional vector bundles do not exhaust the $G$-equivariant $K$-theory of $X$, and that it is necessary to introduce infinite-dimensional vector bundles into the description of $K$-theory (see also \cite{LO}).

One case in which finite-dimensional bundles are sufficient is when the group $G$ is linear (see \cite{Phillips2}), owing to the key fact that in this case every $G$-equivariant vector bundle over $X$ is a direct summand of a $G$-equivariantly trivial bundle, ie. one that is isomorphic to $X\times V$ for some finite-dimensional representation of $G$ on $V$. 

Motivated by this, we give a short proof of Poincar\'{e} duality in this setting. That is, we show that the natural map from $G$-equivariant $K$-theory to $G$-equivariant $K$-homology  (which for us means Baum-Douglas' geometric $K$-homology \cite{BD}), given by pairing with the fundamental class, is an isomorphism for proper cocompact $G$-spin$^c$ manifolds where $G$ is a matrix group.
\begin{theorem}
\label{thm main}
Suppose a linear Lie group $G$ acts properly and cocompactly on a $G$-equivariantly spin$^c$ manifold $X$. Then there is a natural isomorphism
\begin{equation}
\label{eq main}	
\psi\colon K_G^*(X)\to K^{G}_*(X),
\end{equation}
where the left and right-hand sides denote $G$-equivariant $K$-theory and geometric $K$-homology respectively.
\end{theorem}
For compact Lie group actions, Theorem \ref{thm main} is implied by the work of \cite{BOSW} on the isomorphism between the equivariant geometric and analytic models of $K$-homology. Meanwhile, by the Peter-Weyl theorem, such groups form a subclass of linear Lie groups. Consequently, Theorem \ref{thm main} provides another approach to some of the results in \cite{BOSW}.

While Theorem \ref{thm main} makes no reference to the \emph{analytic} model of $K$-homology \cite{BCH,Kasparov}, we note that \eqref{eq main} still holds when the right-hand side is replaced by the analytic $K$-homology group $KK_*^G(C_0(X),\mathbb{C})$. Indeed, as a special case of Emerson-Meyer's general second duality result in \cite[Section 6]{EmersonMeyer}, (in particular, see the first display after (1.5) in \cite{EmersonMeyer}), we have
\begin{equation*}
KK_*^G(C_0(X),\mathbb{C})\cong K_G^*(TX).
\end{equation*}
By the Thom isomorphism from \cite[Theorem 8.11]{Phillips} (see also Theorem \ref{thm thom} below), we have 
\begin{equation*}
K_G^*(TX)\cong K^*_G(X).
\end{equation*}
Putting this together gives $K_G^*(X)\cong KK_*^{G}(C_0(X),\mathbb{C})$. Combined with Theorem \ref{thm main}, this gives:
\begin{corollary}
\label{cor cor1}
Suppose a linear Lie group $G$ acts properly and cocompactly on a $G$-equivariantly spin$^c$ manifold $X$. Then we have an isomorphism between the $G$-equivariant geometric and analytic $K$-homology groups of $X$:
\begin{equation}
\label{eq equivalence}
K^{G}_*(X)\cong KK_*^G(C_0(X),\mathbb{C}).
\end{equation}
In particular, the geometric $K$-homology groups $K_*^G$ are part of a $G$-equivariant extraordinary homology theory.
\end{corollary}
The natural map realizing this isomorphism \eqref{eq equivalence} is the Baum-Douglas map \cite{BD,GMW}, which can be described as follows. Let $(M,E,f)$ be a $G$-equivariant geometric $K$-cycle for $K_*^G(X)$ (see Definition \ref{def geometric} below). Then $M$ is a $G$-equivariantly spin$^c$ manifold with Dirac operator $D_M$ acting on sections of a spinor bundle $S_M$, $E$ is a $G$-equivariant vector bundle over $M$, and $f\colon M\to X$ is a $G$-equivariant continuous map. The Baum-Douglas map takes
$$[M,E,f]\mapsto f_*\left[L^2(S_M\otimes E),\varphi,D_E(1+D_E^2)^{-1/2}\right],$$
where $\varphi$ is the $*$-representation of $C_0(M)$ on $\mathcal{B}(L^2(E))$ given by pointwise multiplication, and the right-hand side is the pushforward under $f$ of a class in $KK_*^G(C_0(M),\mathbb{C})$. Corollary \ref{cor cor1} then implies:
\begin{corollary}
\label{cor cor2}
Suppose a linear Lie group $G$ acts properly and cocompactly on a $G$-equivariantly spin$^c$ manifold $X$. Then for any $e\in KK_*^G(C_0(X),\mathbb{C})$, there exist a $G$-equivariantly cocompact spin$^c$ manifold $M$ and a $G$-equivariant continuous map $f\colon M\to X$ such that $e$ is the pushforward under $f$ of the class of a Dirac-type operator on $M$ in $KK_*^G(C_0(M),\mathbb{C})$. More precisely, there exists a vector bundle $E\to M$ such that
$$e=f_*\left[L^2(S_M\otimes E),\varphi,D_E(1+D_E^2)^{-1/2}\right].$$
\end{corollary}

\vspace{0.2in}
\noindent{\em Acknowledgements.}
The authors would like to thank Peter Hochs and Hang Wang for their useful feedback on this paper.

Hao Guo was partially supported by funding from the Australian Research Council through the
Discovery Project DP200100729, and partially by the National Science Foundation through the NSF DMS-2000082.
Varghese Mathai was partially supported by funding from the Australian Research Council, through the
Australian Laureate Fellowship FL170100020.
\vspace{0.3in}
\section{Preliminaries}
\noindent We begin by recalling the definitions and facts we will need. Unless specified otherwise, $G$ will always denote a closed subgroup of $GL(n,\mathbb{R})$ for some $n$. All vector bundles will be complex. For this section, let $X$ be a locally compact proper $G$-space.  
\vspace{0.1in}
\subsection{Equivariant $K$-theory}
In \cite{Phillips2}, Phillips showed that the $G$-equivariant $K$-theory of the space $X$ with $G$-cocompact supports can be defined in a such a way that is directly analogous to non-equivariant, compactly supported $K$-theory.
\begin{definition}[{\cite[Definition 1.1]{Phillips2}}]
\label{def K-theory}
A $G$-equivariant \emph{$K$-cocycle} for $X$ is a triple $(E,F,t)$ consisting of two finite-dimensional complex $G$-vector bundles $E$ and $F$ over $X$ and a $G$-equivariant bundle bundle map $t:E\to F$ whose restriction to the complement of some $G$-cocompact subset of $X$ is an isomorphism. Two $K$-cocycles $(E,F,t)$ and $(E',F',t')$ are said to be equivalent if there exist finite-dimensional $G$-vector bundles $H$ and $H'$ and $G$-equivariant isomorphisms
$$a:E\oplus H\to E'\oplus H,\qquad b:F\oplus H\to F'\oplus H'$$
such that $b_x^{-1}(t'_x\oplus\id)a_x=t_x\oplus\id$ for all $x$ in the complement of a $G$-cocompact subset of $X$. The set of equivalence classes $[E,F,t]$ of $K$-cocycles forms a semigroup under the direct sum operation, and we define the group $K^0_G(X)$ is the Grothendieck completion of the semigroup of finite-dimensional complex $G$-vector bundles over $X$.
\end{definition}
\begin{remark}
When it is clear from context, or when $X$ is $G$-cocompact, we will omit the map $t$ from the cycle, and simply denote a class in $K^0_G(X)$ by $[E]-[F]$.
\end{remark}

\begin{remark}
For general locally compact groups, Definition \ref{def K-theory} needs to be modified to include infinite-dimensional bundles \cite[Chapter 3]{Phillips}. For $G$ linear, this is not necessary \cite[Theorem 2.3]{Phillips2}.
\end{remark}
\begin{definition}
For each non-negative integer $i$, let 
$$K^i_G(X)=K^0_G(X\times\mathbb{R}^i),$$
where $G$ acts trivially on $\mathbb{R}^i$.
\end{definition}
By \cite[Lemma 2.2]{Phillips2}, $K^i_G$ satisfies Bott periodicity, so that we have a natural isomorphism $K^i_G(X)\cong K^{i+2}_G(X)$ for each $i$. We will use the notation
$$K^*_G(X)=K^0_G(X)\oplus K^1_G(X).$$
In addition, $K^i_G$ are contravariant functors from the category of proper $G$-spaces and proper $G$-equivariant maps to the category of abelian groups, and form an equivariant extraordinary cohomology theory with a continuity property. In particular, Bott periodicity implies that for any $G$-invariant open subset $U\subseteq X$, there is a six-term exact sequence of abelian groups
\begin{equation}
\label{eq sixterm}
	\begin{tikzcd}
	K^0_G(U) \ar{r} & K^0_G(X) \ar{r} & K^0_G(X\backslash U) \ar{d}{\partial} \\
	K^1_G(X\backslash U) \ar{u}{\partial} & K^1_G(X) \ar{l} & K^1_G(U) \ar{l}.
	\end{tikzcd}
\end{equation}
Here, the map $K^i_G(U)\to K^i_G(X)$ is induced by the extension-by-zero homomorphism -- see Remark \ref{rem extbyzero} below, and the boundary maps $\partial$ are defined as in equivariant $K$-theory for compact group actions \cite{Segal}.

\begin{remark}[Extension-by-zero]
\label{rem extbyzero}
Any inclusion of $G$-invariant open subsets $U_1\hookrightarrow U_2$ induces in the obvious way an extension-by-zero $*$-homomorphism $C_0(U_1)\to C_0(U_2)$. This extends to a $*$-homomorphism $C_0(U_1)\rtimes G\to C_0(U_2)\rtimes G$ between crossed products. The induced map on operator $K$-theory, together with the identification $K^i_G(U_j)\cong K_i(C_0(U_j)\rtimes G)$, gives a map
$$K_G^i(U_1)\to K_G^i(U_2).$$
\end{remark}

To prove Poincar\'{e} duality, we will make use of the following Thom isomorphism theorem for $G$-spin$^c$ bundles:
\begin{theorem}[{\cite[Theorem 8.11]{Phillips}}]
\label{thm thom}
Let $E$ be a finite-dimensional $G$-equivariant spin$^c$ vector bundle over $X$. Then there is a natural isomorphism
$$T_G\colon K_G^{i}(X)\xrightarrow{\cong }K_G^{i+\dim E}(E)$$
for $i=0,1$, where $\dim E$ is the real dimension of $E$ and $i+\dim E$ is taken mod $2$.	
\end{theorem}
\vspace{0.1in}
\subsection{The Gysin homomorphism}
\label{subsec gysin}
Theorem \ref{thm thom} can be used to give an explicit geometric description of the Gysin (pushforward) homomorphism in $G$-equivariant $K$-theory, which we will need later. 

Let $Y_1$ and $Y_2$ be two $G$-cocompact $G$-spin$^c$ manifolds and $f\colon Y_1\to Y_2$ a $G$-equivariant continuous map. By cocompactness, both $Y_1$ and $Y_2$ contain only finitely many orbit types. Together with the fact that $G$ is a linear group, this implies, by \cite[Theorem 4.4.3]{palais}, that there exists a $G$-equivariant embedding 
$$j_{Y_1}\colon Y_1\to\mathbb{R}^{2n}$$ 
for some $n$, where $G$ is considered as a subgroup of $GL(2n,\mathbb{R})$. Let 
$$i_{Y_2}\colon Y_2\to Y_2\times\mathbb{R}^{2n}$$ 
denote the zero section, and define the $G$-equivariant embedding
\begin{align*}
i_{Y_1}\colon Y_1&\to Y_2\times\mathbb{R}^{2n},\\
y&\mapsto(f(y),j_{Y_1}(y)).
\end{align*}
Let $\nu_1$ be the normal bundle of $i_{Y_1}$, which we identify with a $G$-invariant tubular neighbourhood $U_1$ of its image. Note that it follows from the two-out-of-three lemma for $G$-equivariant spin$^c$-structures (see \cite[Section 3.1]{PeningtonPlymen} and \cite[Remark 2.6]{HochsMathai}), together with the assumption that $Y_1$ and $Y_2$ are $G$-spin$^c$, that $\nu_1$ has a $G$-spin$^c$ structure, and so Theorem \ref{thm thom} applies. Identifying the normal bundle with $U_1$, we have the Thom isomorphism 
\begin{align}
\label{eq thom}
T_G\colon K_G^{i}(Y_1)&\xrightarrow{\cong}K_G^{i+\dim\nu_1}(U_1),
\end{align}
for $i=0,1$.
The Gysin homomorphism
\begin{equation}
\label{eq gysin}
f_!\colon K^i_G(Y_1)\to K^{i+\dim Y_2-\dim Y_1}_G(Y_2)
\end{equation}
associated to $f$ is then the composition
\begin{equation*}
\begin{tikzcd}
K^i_G(Y_1)\arrow{r}{T_G} & K_G^{i+\dim Y_2+2n-\dim Y_1}(U_1)\arrow{d}{\lambda} \\
&K_G^{i+\dim Y_2+2n-\dim Y_1}(Y_2\times\mathbb{R}^{2n})\arrow{r}{\cong}& K_G^{i+\dim Y_2-\dim Y_1}(Y_2),
\end{tikzcd}
\end{equation*}
where $\lambda$ is induced by the extension-by-zero map associated to the inclusion of $U_1$ into $Y_2\times\mathbb{R}^{2n}$ (see Remark \ref{rem extbyzero} below), and the right horizontal isomorphism is due to Bott periodicity.
\begin{remark}
It can be seen from the above that $f_!$ depends only on the $G$-homotopy class of $f$ and that the Gysin map is functorial under compositions.
\end{remark}
\vspace{0.1in}

\subsection{Equivariant geometric $K$-homology}
We briefly review the equivariant version of Baum and Douglas' geometric definition of $K$-homology \cite{BD}; see \cite{BHS}, \cite{BHS2}, \cite{BOSW}, or \cite{GMW} for more details. As before, $X$ is a locally compact proper $G$-space.
\begin{definition}
\label{def geometric}
A $G$-equivariant \emph{geometric $K$-cycle} for $X$ is a triple $(M,E,f)$, where
\begin{itemize}
\item $M$ is a proper $G$-cocompact manifold with a $G$-equivariant spin$^c$-structure;
\item $E$ is a smooth $G$-equivariant Hermitian vector 	bundle over $M$;
\item $f\colon M\to X$ is a $G$-equivariant continuous map.
\end{itemize}
For $i=0$ or $1$, the \emph{$G$-equivariant geometric $K$-homology group} $K_i^{G}(X)$ is the abelian group generated by geometric $K$-cycles $(M,E,f)$ where $\dim M=i$ mod $2$, subject to the equivalence relation generated by the following three elementary relations:
\begin{enumerate}[(i)]
\item (Direct sum -- disjoint union) For two $G$-equivariant Hermitian vector bundles $E_1$ and $E_2$ over $M$ and a $G$-equivariant continuous map $f\colon M\to X$, 
$$(M\sqcup M,E_1\sqcup E_2,f\sqcup f)\sim(M,E_1\oplus E_2,f);$$
\item (Bordism) Suppose two cycles $(M_1,E_1,f_1)$ and $(M_2,E_2,f_2)$ are \emph{bordant}, so that there exists a $G$-cocompact proper $G$-spin$^c$ manifold $W$ with boundary, a smooth $G$-equivariant Hermitian vector bundle $E\to W$ and a continuous $G$-equivariant map $f\colon W\to X$ such that
	$(\partial W,E|_{\partial W},f|_{\partial W})$ is isomorphic to $(M_1,E_1,f_1)\sqcup(-M_2,E_2,f_2),$ where $-M_2$ denotes $M_2$ with the opposite $G$-spin$^c$ structure. Then
	$$(M_1,E_1,f_1)\sim(M_2,E_2,f_2);$$
\item (Vector bundle modification) Let $V$ be a $G$-spin$^c$ vector bundle of real rank $2k$ over $M$. Upon choosing a $G$-invariant metric on $V$, let $\widehat{M}$ be the unit sphere bundle of $(M\times\mathbb{R})\oplus V$, where the bundle $M\times\mathbb{R}$ is equipped with the trivial $G$-action. Let $F$ be the \emph{Bott bundle} over $\widehat{M}$, which is fibrewise the non-trivial generator of $K^0(S^{2k})$. (See \cite[Section 3]{BHS2} for a more detailed description.) Then 
	$$(M,E,f)\sim(\widehat{M},F\otimes\pi^*(E),f\circ\pi),$$ 
	where $\pi:\widehat{M}\to M$ is the canonical projection.
\end{enumerate}
Addition in $K_i^{\textnormal{geo},G}(X)$ is given by
$$[M_1,E_1,f_1]+[M_2,E_2,f_2]=[M_1\sqcup M_2,E_1\sqcup E_2,f_1\sqcup f_2],$$
the additive inverse of $[M,E,f]$ is its opposite $[-M,E,f]$, while the additive identity is given by the empty cycle where $M=\emptyset$.
\end{definition}
\begin{remark}
\label{rem difference cycles}
The above definition of classes $[M,E,f]$ continues to make sense if we replace the bundle $E$ by a $K$-theory class. Indeed, if 
$$e=[E_1]-[E_2]=[E_1']-[E_2']\in K^0_G(M),$$
then there exists a $G$-vector bundle $F$ over $X$ such that
$$E_1\oplus E_2'\oplus F\cong E_1'\oplus E_2\oplus F.$$
By Definition \ref{def geometric} (i), this means
$$[M,E_1,f]+[M,E_2',f]+[M,F,f]=[M,E_1',f]+[M,E_2,f]+[M,F,f].$$
Adding the inverse of $[M,F,f]$ to both sides and rearranging shows that the class
\begin{align*}
[M,e,f]\coloneqq [M,E_1,f]-[M,E_2,\tn{id}]&=[M,E_1',\tn{id}]-[M,E_2',\tn{id}]
\end{align*}
is well-defined.
\end{remark}

Finally, we can describe vector bundle modification using the Gysin homomorphism \ref{eq gysin}. To do this, let $\widehat{M}$ be the manifold underlying the vector bundle modification of a cycle $(M,E,f)$ by a bundle $V$, as in Definition \ref{def geometric} (iii). Then $\widehat{M}$ is the unit sphere bundle of $(M\times\mathbb{R})\oplus V$. We will refer to the $G$-equivariant embedding
\begin{align*}
s\colon M&\to\widehat{M}\subseteq(M\times\mathbb{R})\oplus V,\\
m&\mapsto(m,1,0)
\end{align*}
as the \emph{north pole section}.
\begin{lemma}
\label{lem modification}
Let $(M,E,f)$ be a geometric cycle for $X$. Let $(\widehat{M},F\otimes\pi^*(E),f\circ \pi)$ be its modification by a $G$-spin$^c$ vector bundle $V$ of even real rank, and let $\pi\colon\widehat{M}\to M$ be the projection. Let $s\colon M\to\widehat{M}$ be the north pole section. Then
$$(\widehat{M},F\otimes\pi^*(E),f\circ \pi)\sim(\widehat{M},s![E],f\circ \pi).$$
\end{lemma}
\begin{proof}
The proof we give is similar to that of \cite[Lemma 3.5]{BOSW} concerning the case of compact Lie group actions; compare also the discussion following \cite[Definition 6.9]{BHS2}. To begin, observe that the total space of $V$ can be identified $G$-equivariantly with a $G$-invariant tubular neighbourhood $U$ of the embedding $s\colon M\to\widehat{M}$. The Gysin map $s_!$ is then the composition
\begin{equation}
\label{eq thom mod}
K_G^*(M)\xrightarrow{T_G}K_G^*(U)\xrightarrow{\lambda}K_G^*(\widehat{M}),
\end{equation}
where $T_G$ is the Thom isomorphism in the form \eqref{eq thom}, while $\lambda$ is the homomorphism induced by the extension-by-zero map $C_0(U)\to C_0(\widehat{M})$. Note that $T_G$ is essentially given via tensor product with a ``Bott element", and \eqref{eq thom mod} admits the following geometric description. Let $F$ be the Bott bundle over $\widehat{M}$, and let $F_0$ be the bundle defined by pulling back the restriction $F|_M$ along $\pi$. The composition \eqref{eq thom mod} is then given by pulling back a vector bundle over $M$ along $\pi$ and tensoring with the class $[F]-[F_0]$. On the other hand, since $\widehat{M}$ is the boundary of the unit sphere bundle of $(M\times\mathbb{R})\oplus W$, and the bundle $F_0$ is pulled back from $M$, the cycle $(\widehat{M},F_0\otimes\pi^*(E),f\circ \pi)$ is bordant to the empty cycle. Thus we have a bordism of cycles
$$(\widehat{M},F\otimes\pi^*(E),f\circ \pi)\sim(\widehat{M},([F]-[F_0])\otimes\pi^*(E),f\circ \pi),$$
whence the right-hand side is equal to $(\widehat{M},s![E],f\circ\pi)$ by the description of the composition \eqref{eq thom mod} given above.
\end{proof}


\vspace{0.2in}
\section{Poincar\'{e} duality for geometric $K$-homology}
\noindent In this section we prove Theorem \ref{thm main}. For the rest of this section, let $X$ be a proper $G$-spin$^c$ manifold with $X/G$ is compact. 



We can define the following natural map between $K^*_G(X)$ and $K_*^G(X)$, which can be thought of as cap product with the fundamental $K$-homology class on $X$. For this, let $S^1$ be the unit circle in $\mathbb{C}$, and define the map 
\begin{align}
\label{eq c}
c\colon X&\to X\times S^1\nonumber\\
x&\mapsto(x,1).	
\end{align}
\begin{definition}
\label{def phi}
Define $\phi\colon K_G^*(X)\to K_*^G(X)$ by
\begin{align*}
\phi\colon K^i_G(X)&\to K_{i+\dim X}^G(X),\\
x&\mapsto
\begin{cases}
	[X,e,\textnormal{id}],&\textnormal{ if }i=0,\\
	[X\times S^1,c_!(e),\textnormal{pr}_1],&\textnormal{ if }i=1,
\end{cases}
\end{align*}
where $\textnormal{pr}_1\colon X\times S^1$ is the projection onto the first factor, and we have used the notation from Remark \ref{rem difference cycles}.
\end{definition}

We now show that $\phi$ is an isomorphism by defining explicitly a map $\psi$ that will turn out to be its inverse.
\begin{definition}
\label{def psi}
Define the map $\psi\colon K_*^G(X)\to K_G^*(X)$ by
\begin{align*}
\psi\colon K_{i+\dim X}^{G}(X)&\to K^i_G(X),\\
[M,E,f]&\mapsto f_![E],
\end{align*}
for $i=0,1$, where $f_!$ is the Gysin homomorphism from \eqref{eq gysin}.
\end{definition}
\begin{remark}
Note that
$$f_![E]\in f_!(K_G^0(M))\subseteq K_G^{\dim X-\dim M}(X)=K^{\dim X-(i+\dim X)}_G(X)=K^{i}_G(X),$$ 
so the degrees make sense.	
\end{remark}

We first need to show that $\psi$ is well-defined. For this, we will use the following:
\begin{lemma}
\label{lem gysin bordism}
	Let $W$ be a $G$-cocompact, $G$-spin$^c$ manifold-with-boundary, and let $X$ be a $G$-cocompact $G$-spin$^c$ manifold. Let $h\colon W\to X$ be a $G$-equivariant map, and let $i\colon\partial W\hookrightarrow W$ be the natural inclusion. Then the composition 
\begin{equation*}
K^*_G(W)\xrightarrow{i^*}K^*_G(\partial W)\xrightarrow{(h|_{\partial W})_!}K^*_G(X)
\end{equation*}
is the zero map.
\end{lemma}
\begin{proof}
We give the proof when $\dim\partial W=\dim X$ mod $2$ and show that
\begin{equation}
\label{eq zero composition}
K^0_G(W)\xrightarrow{i^*}K^0_G(\partial W)\xrightarrow{(h|_{\partial W})_!}K^0_G(X)
\end{equation}
is the zero map; the proofs for the other cases are similar.

Let us consider the composition \eqref{eq zero composition} upon applying the Thom isomorphism, Theorem \ref{thm thom}, in the form of \eqref{eq thom}, and use the description of the Gysin map from \eqref{eq gysin}. 

Let $j_W$ be a $G$-equivariant embedding of $W$ into $\mathbb{R}^{2n}$ for some $n$, where $G$ is realized as a subgroup of $GL(2n,\mathbb{R})$; note that this is possible because $G$ is assumed to be linear. Let $j_{\partial W}$ denote the restriction of $j_W$ to $\partial W$. Let $i_X\colon X\to X\times\mathbb{R}^{2n}$ be the zero section. Define the embedding
\begin{align*}
i_W\colon W&\to X\times\mathbb{R}^{2n},\\
w&\mapsto\left(h(w),j_W(w)\right),
\end{align*}
and let $i_{\partial W}$ be the restriction of $i_W$ to $\partial W$. 
Let $\nu_W$ and $\nu_{\partial W}$ denote the respective normal bundles of the embeddings $i_W$ and $i_{\partial W}$. We may identify these normal bundles with $G$-invariant tubular neighbourhoods $U_W$ and $U_{\partial W}$ in $X\times\mathbb{R}^{2n}$, noting that in general $U_W$ has boundary. Since the normal bundle of $\partial W$ in $W$ is trivial and one-dimensional, there is a natural $G$-equivariant identification 
\begin{equation}
\label{eq suspension identification}
U_{\partial W}\cong\partial U_{W}\times(-\varepsilon,\varepsilon)
\end{equation}
for some $\varepsilon>0$.
It follows from the two-out-of-three lemma for $G$-equivariant spin$^c$-structures (see \cite[Section 3.1]{PeningtonPlymen} and \cite[Remark 2.6]{HochsMathai}), together with the fact that $W$ and $X$ are $G$-spin$^c$, that $\nu_W$ and $\nu_{\partial W}$ are $G$-spin$^c$ vector bundles, and hence Theorem \ref{thm thom} applies. The resulting Thom isomorphisms for $W$, $\partial W$, and $X$ (in the notation of \eqref{eq thom}) are shown as vertical arrows in the following commutative diagram:
\begin{equation}
\begin{tikzcd}
\label{eq thom diagram}
K^0_G(W)\arrow{r}{i^*}\arrow{d}{T_G} & K^0_G(\partial W)\arrow{r}{(h|_{\partial W})_!}\arrow{d}{T_G} & K^0_G(X)\arrow{d}{T_G} \\
K^1_G(U_W)\arrow{r}{T_Gi^*} & K^0_G(U_{\partial W})\arrow{r}{\lambda} &K^0_G(X\times\mathbb{R}^{2n}),
\end{tikzcd}
\end{equation}
where the map $T_Gi^*$ is determined uniquely by commutativity, and the homomorphism $\lambda$ is induced by the extension-by-zero map $C_0(U_{\partial W})\to C_0(X\times\mathbb{R}^{2n})$ as in Remark \ref{rem extbyzero}. It thus suffices to show that the composition $\lambda\circ T_Gi^*$ vanishes. 

By \cite[Lemma 2.2]{Phillips2} or \cite[Chapter 5]{Phillips}, we have a six-term exact sequence
\begin{equation}
\label{eq sixterm}
\begin{tikzcd}
K^0_G(U_{W}\backslash\partial U_{W})\arrow{r}{\lambda} & K^0_G(U_W)\arrow{r}{j^*} & K^0_G(\partial U_W)\arrow{d}{\delta} \\
K^1_G(\partial U_W)\arrow{u}{\delta} & K^1_G(U_{W})\arrow{l}{j^*} & K^1_G(U_{W}\backslash\partial U_{W})\arrow{l}{\lambda},
\end{tikzcd}
\end{equation}
where $j^*$ is induced by the inclusion $j\colon\partial U_W\hookrightarrow U_W$ and the maps $\lambda$ are again induced by the extension-by-zero map $C_0(U_W\backslash\partial U_{W})\to C_0(U_W)$. The identification \eqref{eq suspension identification} gives a natural isomorphism $K_G^1(\partial U_W)\cong K_G^0(U_{\partial W})$. Using this, the bottom row of \eqref{eq thom diagram} fits into the following commutative diagram:
\begin{equation}
\label{eq suspension diagram}
\begin{tikzcd}
K^1_G(U_W)\arrow{r}{T_Gi^*}\arrow{dr}{j^*} & K^0_G(U_{\partial W})\arrow{r}{\lambda} & K^0_G(U_{X})\\
&K^1_G(\partial U_W)\arrow{r}{\delta}\arrow{u}{\cong} & K^0_G(U_W\backslash U_{\partial W})\arrow{u}{\lambda}.
\end{tikzcd}
\end{equation}
It follows from exactness of \eqref{eq sixterm} that $\lambda\circ T_Gi^*=0$, and hence $(h|_{\partial W})_!\circ i^*=0$.
\end{proof}
%
%

\begin{proposition}
The map $\psi$ is well-defined.
\end{proposition}
\begin{proof}
That $\psi$ respects disjoint union/direct sum is clear, since for any element of the form $[M,E_1\oplus E_2,f]\in K^G_*(M)$, we have
$$\psi[M,E_1\oplus E_2,f]=f_![E_1\oplus E_2]=f_![E_1]+f_![E_2]\in K^*_G(M).$$
Next, let $[W,E,h]$ be an equivariant bordism between two elements $[M_1,E_1,h_1]$ and $[-M_2,E_2,h_2]$. Then Lemma \ref{lem gysin bordism} applied to $\partial W=M_1\sqcup -M_2$ implies that 
$$(h_1)_![E_1]=(h_2)_![E_2],$$
hence $\psi$ is well-defined with respect to the bordism relation. To see that $\psi$ is well-defined with respect to vector bundle modification, let $(\widehat{M},F\oplus\pi^*(E),f\circ\pi)$ be the modification of a cycle $(M,E,f)$ for $X$ by a bundle $V$, as in Definition \ref{def geometric} (iii). By Lemma \ref{lem modification}, we have
$$[\widehat{M},F\oplus\pi^*(E),f\circ\pi]=[\widehat{M},s_![E],f\circ\pi].$$
Functoriality of the Gysin map with respect to composition, together with the fact that $\pi\circ s=\textnormal{id}$, now implies
\begin{align*}
\psi[\widehat{M},s_![E],f\circ\pi]&=	(f\circ\pi)_! s_![E]\\
&=f_!\circ(\pi\circ s)_![E]\\
&=f_![E]\\
&=\psi[M,E,f].\qedhere
\end{align*}
\end{proof}
\begin{proposition}
\label{prop injective}
The map $\phi$ is injective.	
\end{proposition}
\begin{proof}
For any
$e\in K_G^i(X)$, we have
\begin{align*}
\psi\circ\phi(e)&=
\begin{cases}	
\psi[X,e,\textnormal{id}]=\textnormal{id}_!(e),&\textnormal{if $i=0$},\\
\psi[X\times S^1,c_!(e),\textnormal{pr}_1]=(\textnormal{pr}_1)_!(\textnormal{id}_!(e))=(pr_1\circ\textnormal{id})_!(e),&\textnormal{if $i=1$},
\end{cases}
\end{align*}
which are both equal to $e$, where we have used functoriality of the Gysin map. Hence $\psi\circ\phi=\textnormal{id}$, so $\phi$ is injective. 
\end{proof}

For surjectivity, we will use the Gysin homomorphism from subsection \ref{subsec gysin}, together with the following result, which is a special case of \cite[Theorem 4.1]{BOSW} but applied to linear instead of compact $G$.
\begin{lemma}
\label{lem gysin modification}
Let $M,N,X$ be three $G$-cocompact $G$-spin$^c$ manifolds, $g\colon N\to X$ a $G$-equivariant continuous map, and $f\colon M\to N$ a $G$-equivariant embedding with even codimension. Then for any $G$-vector bundle $E\to M$, we have
$$[M,E,g\circ f]=[N,f![E],g]\in K_*^G(X).$$
\end{lemma}
\begin{proof}
The proof of Theorem \cite[Theorem 4.1]{BOSW}, which was stated for compact Lie groups, goes through with no changes to our setting.
\end{proof}

\begin{proposition}
\label{prop phi surj}
The map $\phi$ is surjective.
\end{proposition}
\begin{proof}
Examining Definitions \ref{def phi} and \ref{def psi}, one sees that $\phi\circ\psi$ is given, at the level of geometric cycles, by
$$
[M,E,f]\mapsto f_![E]\mapsto
\begin{cases}
	[X,f_![E],\textnormal{id}],&\textnormal{ if }\dim M=\dim X\mod{2},\\
	[X\times S^1,c_!(f_![E]),\textnormal{pr}_1],&\textnormal{ otherwise,}
\end{cases}
$$
where the map $c$ was defined in \eqref{eq c}. Let $i\colon M\to\mathbb{R}^{2n}$ be a $G$-equivariant embedding for some $n$, and let $j\colon X\to\mathbb{R}^{2n}\times X$ be the zero section. Upon compactifying $\mathbb{R}^{2n}$, $f$ factors as
\begin{equation}
\begin{tikzcd}
 & S^{2n}\times X\arrow{d}{\textnormal{pr}_2}\\
M\arrow{ur}{i\times f}\arrow{r}{f}&X,
\end{tikzcd}
\end{equation}
where $\textnormal{pr}_2$ is the projection onto the second factor, and $j$ becomes an embedding $X\to S^{2n}\times X$. 

Suppose first that $\dim M=\dim X\mod{2}$. Then it suffices to prove that any geometric cycle of the form $(M,E,f)$ is equivalent to $(X,f_![E],\textnormal{id})$. By Lemma \ref{lem gysin modification} applied to the embedding $i\times f$, we have
\begin{equation}
\label{eq mod 1}
[M,E,f]=[S^{2n}\times X,(i\times f)_![E],\textnormal{pr}_2]\in K_*^G(X).	
\end{equation}
Meanwhile, Lemma \ref{lem gysin modification} applied to $j$, together with functoriality of the Gysin map, yields
\begin{align}
\label{eq mod 2}
[X,f_![E],\textnormal{id}]&=[X,f_![E],\textnormal{pr}_2\circ j]\nonumber\\
&=[S^{2n}\times X,j_!(f_![E]),\textnormal{pr}_2]\nonumber\\
&=[S^{2n}\times X,(j\circ f)_![E],\textnormal{pr}_2].
\end{align}
Finally, the maps $i\times f$ and $j\circ f$ are $G$-homotopic through
\begin{align*}
F\colon M\times I&\to\mathbb{R}^{2n}\times X\hookrightarrow S^{2n}\times X,\\
(m,t)&\mapsto((1-t)i(m),f(m)).
\end{align*}
Invariance of the Gysin map under $G$-homotopy now implies that \eqref{eq mod 1} and \eqref{eq mod 2} are equal, and we conclude.

The case of $\dim M\neq\dim X\Mod{2}$ proceeds analogously as follows. Here, we need to show that any geometric cycle of the form $(M,E,f)$ is equivalent to $(X\times S^1,c_!(f_![E]),\textnormal{pr}_1)$, where $c$ was defined in \eqref{eq c}. Similar to \eqref{eq mod 1}, we have
\begin{equation}
\label{eq mod 3}
[M,E,f]=[S^{2n}\times X\times S^1,(i\times(c\circ f))_![E],\textnormal{pr}_{2}]\in K_*^G(X),
\end{equation}
Define
\begin{align*}
\tilde{j}\colon X\times S^1&\to S^{2n}\times X\times S^1,\\
(x,s)&\mapsto (j(x),s).
\end{align*}
Then
\begin{align}
\label{eq mod 4}
[X\times S^1,c_!(f_![E]),\textnormal{pr}_1]&=[X\times S^1,c_!(f_![E]),\textnormal{pr}_2\circ\tilde{j}_!]\nonumber\\
&=[S^{2n}\times X\times S^1,\tilde{j}_!(c_!(f_![E]),\textnormal{pr}_2]\nonumber\\
&=[S^{2n}\times X,(\tilde{j}\circ c\circ f)_![E],\textnormal{pr}_2].
\end{align}
The maps $i\times(c\circ f)$ and $\tilde{j}\circ c\circ f$ are $G$-homotopic through
\begin{align*}
F\colon M\times I&\to\mathbb{R}^{2n}\times X\times S^1\hookrightarrow S^{2n}\times X\times S^1,\\
(m,t)&\mapsto((1-t)i(m),f(m),1),
\end{align*}
and the claim follows from $G$-homotopy invariance of the Gysin map.
\end{proof}
Propositions \ref{prop phi surj} and equation \eqref{prop injective} together imply that $\phi\colon K^i_G(X)\to K^G_{i+\dim X}(X)$ is an isomorphism for $i=0,1$, which establishes Theorem \ref{thm main}.

\vspace{0.3in}
\bibliographystyle{plain}

\bibliography{../BigBibliography/mybib}

\end{document}